\documentclass{amsart}

\usepackage{amsmath,amssymb,parskip}
\usepackage{hyperref}
\usepackage{todonotes}

\usepackage{bbm}
\usepackage{graphicx}
\usepackage[all]{xy}
\usepackage{enumerate}
\usepackage{tikz}
\usepackage{float}
\usepackage{psfrag}
\usepackage{upgreek}

\newtheorem{theo}{Theorem}[section]

\newtheorem{prop}[theo]{Proposition}
\newtheorem{lemm}[theo]{Lemma}

\newtheorem{rem}[theo]{Remark}

\newenvironment{rema}{\begin{rem}\rm }{\hfill $\blacktriangleleft$\end{rem}}

\def\bb#1{\mathbb{#1}} \def\bbcp{\mathbb{C}\mathbb{P}} \def\m#1{\mathcal{#1}}

 \def\co{\colon\thinspace}

\def\ham#1{\mathrm{Ham}#1}
\def\Piuniv{\Pi^{\mathrm{univ}}}

\def\QH{{QH}}
\def\SR{{SR}}
\def\Q{\mathbb Q}
\def\N{\mathbb N}
\def\R{\mathbb R}
\def\C{\mathbb C}
\def\Z{\mathbb Z}
\def\F{\mathbb F}
\def\X{\mathbb X}
\def \bbcp{\mathbb C\mathbb P}

\newcommand{\Lin}{\mathrm{Lin}}

\newcommand{\PbP}{ \bbcp^2\#\,\overline{ \bbcp}\,\!^2}

\newcommand{\Mucc}{M_{\mu,c_1,c_2}}

\newcommand{\omucc}{\omega_{\mu,c_1,c_2}}

\newcommand*{\quot}[2]%
{\ensuremath{%
   \raisebox{.35ex}{\ensuremath{#1}}\big/\raisebox{-.35ex}{\ensuremath{#2}}}}

\begin{document}

\title[Hamiltonian loops in the kernel of Seidel's representation]{Non-contractible Hamiltonian loops in the kernel of Seidel's representation}

\author{S\'ilvia Anjos and R\'emi Leclercq}

\date{\today}

\thanks{The authors would like to thank Dusa McDuff for her interest and useful discussions. 
The first author is partially funded by FCT/Portugal through project UID/MAT/04459/2013 and project EXCL/\-MAT-GEO/\-0222/\-2012. 
The second author is partially supported by ANR Grant ANR-13-JS01-0008-01.}

\address{SA: Center for Mathematical Analysis, Geometry and Dynamical Systems, Mathematics Department,
Instituto Superior T\'ecnico, Av. Rovisco Pais, 1049-001 Lisboa, Portugal}
\email{sanjos@math.ist.utl.pt}
\address{RL: Laboratoire de Math\'ematiques d'Orsay, Univ. Paris-Sud, CNRS, Universit\'e Paris-Saclay, 91405 Orsay, France.}
\email{remi.leclercq@math.u-psud.fr}

\subjclass[2010]{Primary 53D45; Secondary 57S05, 53D05}
\keywords{symplectic geometry, Seidel morphism, toric symplectic manifolds, Hirzebruch surfaces}

\begin{abstract}
The main purpose of this note is to exhibit a Hamiltonian diffeomorphism loop undetected by the Seidel morphism of a 1--parameter family of 2--point blow-ups of $S^2 \times S^2$, exactly one of which being monotone. As side remarks, we show that Seidel's morphism is injective on all Hirzebruch surfaces, and discuss how to adapt the monotone example to the Lagrangian setting.
\end{abstract}

\maketitle

\section{Introduction}\label{sec:introduction}

The motivation for this work is the search of homotopy classes of loops of Hamiltonian diffeomorphisms which are not detected by Seidel's morphism. Given a symplectic manifold $(M,\omega)$, and its Hamiltonian diffeomorphism group $\ham(M,\omega)$, recall that Seidel's morphism
\begin{align*}
  \m S \co \pi_1(\ham(M,\omega)) \rightarrow QH_*(M,\omega)^\times
\end{align*}
was defined on a covering of $\pi_1(\ham(M,\omega))$ by Seidel in \cite{Seidel97} for strongly semi-positive symplectic manifolds and then on the fundamental group itself and for any closed symplectic manifold by Lalonde--McDuff--Polterovich in \cite{LMcDP99}.

The target space, $QH_*(M,\omega)^\times$, is the group of invertible elements of the quantum homology of $(M,\omega)$. More precisely, the small quantum homology of $(M,\omega)$ is $QH_*(M,\omega) = H_*(M;\bb Z) \otimes \Pi$ where $\Pi = \Piuniv[q,q^{-1}]$ with $q$ a degree 2 variable and the ring $\Piuniv$ consisting of generalized Laurent series in a degree 0 variable $t$:
\begin{align}
  \label{eq:PiUniv}
   \Pi^{\mathrm{univ}}:= \left\{ \left. \sum_{\kappa \in \R} r_\kappa t^\kappa \,\right|\, r_\kappa \in \Q, \mbox{ and } \forall c \in \bb R, \, \#\{ \kappa > c\mid r_\kappa \neq 0\} < \infty \right \} \,.
\end{align}

Since its construction, Seidel's morphism has been successfully used to detect many Hamiltonian loops (see e.g \cite{McDuff10} and references therein), and was extended or generalized to various situations (see e.g \cite{Hutchings}, \cite{Savelyev}, \cite{HuLalonde}, \cite{HuLalondeLeclercq}, \cite{FOOO11}). A particular extension consists of secondary-type invariants, whose construction is based on Seidel's construction after enriching Floer homology by considering Leray--Serre spectral sequences introduced by Barraud--Cornea \cite{BarraudCornea07}, and which should detect loops undetected by Seidel's morphism \cite{BarraudCornea}. However, there were no Hamiltonian loops with non-trivial homotopy class known to be undetected by Seidel's morphism (as far as we know). This short note intends to provide the first example of such a loop on a family of symplectic manifolds. Moreover, the example is explicit and thus can easily be used to test other constructions. Notice finally that this example can also be used to construct other examples (e.g by products, see \cite{L09}).

\subsection*{First try: Symplectically aspherical manifolds}\label{sec:first-try}

Looking for elements in the kernel of the Seidel morphism, one might first consider symplectically aspherical manifolds, by which we mean that both the symplectic form and the first Chern class vanish on the second homotopy group of the manifold. Indeed, such manifolds have trivial Seidel morphism.

The geometric reason for this is that, by construction, the Seidel morphism of $(M,\omega)$ counts pseudo-holomorphic section classes of a fibration over $S^2$ with fiber $(M,\omega)$. The difference between two such classes is thus given by elements of $\pi_2(M)$ admitting a pseudo-holomorphic representative, whose existence is prevented by symplectic asphericity.

Alternatively, this can be proved via purely algebraic methods, using the equivalent description of Seidel's morphism, as a representation of $\pi_1(\ham(M,\omega))$ into the Floer homology of $(M,\omega)$. Given a loop of Hamiltonian diffeomorphisms, one gets an automorphism of $HF_*(M,\omega)$ which can be shown to act trivially by playing around with the following facts:
\begin{enumerate}[(i)]
\item Morse homology (the quantum homology of symplectically aspherical manifolds) is a ring over which Floer homology is a module.
\item All involved morphisms (PSS, Seidel, continuation) are module morphisms.
\item Any automorphism of Morse homology preserves the fundamental class, since it generates the top degree homology group.
\item The fundamental class is the unit of the Morse homology ring.
\end{enumerate}
This line of ideas, which goes back to Seidel, has been used by McDuff--Salamon in \cite{McDuffSalamon04} to simplify Schwarz's original proof of invariance of spectral invariants. It has then been adapted by Leclercq in \cite{L08} to Lagrangian spectral invariants and to prove the triviality of the relative (i.e Lagrangian) Seidel morphism by Hu--Lalonde--Leclercq in \cite{HuLalondeLeclercq} (see Lemma 5.5).

Now, even though aspherical manifolds seem to be ideal candidate, there are no homotopically non-trivial loops of Hamiltonian diffeomorphisms known to the authors in such manifolds...

\subsection*{Second try: Symplectic toric manifolds}\label{sec:second-try}

Symplectic toric geometry provides a large class of natural examples of symplectic manifolds which are complicated enough to be interesting while simple enough so that many rather involved constructions can be explicitly performed. In \cite{AnjosLeclercq14}, we computed the Seidel morphism on NEF toric 4--manifolds following work of McDuff and Tolman \cite{McDuffTolman06}. Recall that by definition $(M,J)$ is a NEF pair if there are no $J$--pseudo-holomorphic spheres in $M$ with negative first Chern number.
This gave, \textit{in the particular case of 4--dimensional toric manifolds}, an elementary and somehow purely symplectic way to perform these computations previously obtained by Chan, Lau, Leung, and Tseng \cite{CLLT} (and using works by Fukaya, Oh, Ohta, and Ono \cite{FOOO10}, and Gonz\'alez and Iritani \cite{GonzalesIritani11}). We also showed that one could then deduce the Seidel morphism of some non-NEF symplectic manifolds and, as an example, we made explicit computations for some Hirzebruch surface.

The easiest symplectic toric 4--manifolds for which we can exhibit a non-trivial element in the kernel of the Seidel morphism are 2--point blow-ups of $S^2 \times S^2$. More precisely, start with the monotone product $(S^2 \times S^2, \omega_1)$\footnote{traditionally, $\omega_\mu$ denotes the product symplectic form with total area $\mu \geq 1$ on the first factor and area $ 1$ on the second one} on which we perform two blow-ups. Notice that the resulting symplectic manifold is monotone only when the respective sizes of the blow-ups coincide \textit{and are equal to $\tfrac{1}{2}$}.

In Section \ref{sec:2-point-blowup}, we exhibit a specific loop of Hamiltonian diffeomorphisms whose homotopy class is in the kernel of Seidel's morphism if and only if the size of the two blow-ups coincide. Since this loop, obtained from two circle actions, can easily be seen to be non-trivial (in \cite{AnjosPinsonnault13}, Anjos and Pinsonnault computed the rational homotopy of symplectomorphism groups of these manifolds), this obviously yields a family of symplectic manifolds, only one of which being monotone, with non-injective Seidel morphism, i.e
\begin{theo}\label{theorem:3pt-blowup}
The Seidel morphism of the 2--point blow-ups of  $(S^2 \times S^2, \omega_1)$ with blow-ups of equal (arbitrary) sizes is not injective. 
\end{theo}
 
In our search for undetected Hamiltonian loops, we realized that 
\begin{theo}\label{theorem:Hirzebruch}
  Seidel's morphism is injective on all Hirzebruch surfaces.
\end{theo}
While this is not hard to prove and might be well-known to experts, we did not find it in the literature and thus include a proof in Section \ref{sec:hirzebruch-surfaces}.

\subsection*{Discussion on the adaptation to the Lagrangian setting}\label{sec:disc-adapt-lagr}

As mentionned above, there is a relative (i.e Lagrangian) version of the Seidel morphism defined by Hu--Lalonde in \cite{HuLalonde} and further studied by Hu--Lalonde--Leclercq in \cite{HuLalondeLeclercq}. There are two ways to adapt the example of Theorem \ref{theorem:3pt-blowup} to the Lagrangian setting which we discuss here. (However, in order to keep this note short -- and without too many technical details on the standard tools involved here --, we will not investigate these ideas further on here.)

First, let us remark that to get the Lagrangian version of the Seidel morphism, we need to consider a monotone Lagrangian of minimal Maslov at least 2. So, in what follows, we have in mind the only monotone symplectic manifold of the family mentionned above, i.e the monotone product $S^2 \times S^2$ with area of each factor equals to 1 on which we perform two blow-ups of size $\tfrac{1}{2}$.

$\bullet$~\textit{The first way to relate absolute and relative settings} is to consider the diagonal of the symplectic product. More precisely, let $(M,\omega)$ be a monotone symplectic manifold. The diagonal $\Delta \simeq M$ is a monotone Lagrangian of the product $(M \times M,\omega \oplus (-\omega))$, which we denote $(\widehat M,\widehat \omega)$ for short, with minimal Maslov number equal to twice the minimal first Chern number of $(M, \omega)$ and thus greater than or equal to 2. This allows us to consider the Lagrangian Seidel morphism:
\begin{align*}
  \m S_\Delta \co \pi_1(\ham(\widehat M,\widehat \omega),\ham_\Delta(\widehat M,\widehat \omega)) \rightarrow QH_*(\Delta)^\times 
\end{align*}
where $\ham_\Delta$ denotes the subgroup of $\ham$ formed by Hamiltonian diffeomorphisms which preserve $\Delta$ and $QH_*(\Delta)$ denotes the Lagrangian quantum homology of $\Delta$.

An element $\phi \in \pi_1(\ham(M,\omega))$ generated by the Hamiltonian $H \co M \times [0,1] \rightarrow \bb R$, induces $\widehat \phi \in \pi_1(\ham(\widehat M,\widehat \omega),\ham_\Delta(\widehat M,\widehat \omega))$, generated by $\widehat F = F \oplus 0 \co \widehat M \times [0,1] \rightarrow \bb R$. To get an element in the kernel of the Lagrangian Seidel morphism, it only remains to prove that:
\begin{enumerate}[(i)]
\item $\m S(\phi) = \m S_\Delta (\widehat \phi)$ in $QH_*(M,\omega) \simeq QH_*(\Delta)$, $\;$ and $\quad$ (ii)~$\widehat \phi$ is non zero.
\end{enumerate}
Note that in (i), not only the quantum homologies are canonically identified but the chain complexes themselves coincide and this identification agrees with the PSS morphisms in the following sense:
\begin{align*}
  \xymatrix{\relax
    QH_*(M,\omega) \ar@{=}[r] \ar[d]_{\mathrm{PSS}} & QH_*(\Delta) \ar[d]^{\mathrm{PSS}} \\
    HF_*(H,J) \ar@{=}[r] & HF_*(\widehat H, \widehat J :\Delta) 
  }
\end{align*}
as proved in the monotone setting by Leclercq--Zapolsky in \cite{LeclercqZapolsky} ($J$ denotes an almost complex structure on $M$, compatible with and tamed by $\omega$, while $\widehat J$ denotes an almost complex structure on $\widehat M$ adapted to $J$). This makes us believe that (i) can be straigthforwardly shown to hold.

On the other hand, proving (ii) will require some other technique.

$\bullet$~\textit{The second way to the Lagrangian setting} is to use Albers's comparison map between Hamiltonian and Lagrangian Floer homologies from \cite{Albers}, denoted below by $\m A$, which relates the absolute and relative Seidel morphisms via the following commutative diagram (see \cite{HuLalonde}):
\begin{align*}
  \xymatrix{\relax
    \pi_1(\ham(M,\omega)) \ar[r] \ar[d]_{\m S} & \pi_1(\ham(M,\omega),\ham_L(M,\omega)) \ar[r] \ar[d]^{\m S_L} & \pi_0(\ham_L(M,\omega)) \\
    HF_*(M,\omega) \ar[r]_{\m A} & HF_*(M,\omega;L) &
  }
\end{align*}
where $L$ is a closed monotone Lagrangian of $(M,\omega)$ with minimal Maslov number at least 2.

To get an interesting example via this method, one has to choose $L$ such that $HF_*(M,\omega;L) \neq 0$ and to prove (again) that the image of $\phi \in \pi_1(\ham(M,\omega))$ in $\pi_1(\ham(M,\omega),\ham_L(M,\omega))$ is non-trivial.


\section{Background and user manual for Sections \ref{sec:hirzebruch-surfaces} and \ref{sec:2-point-blowup}}
\label{sec:background}

In order to prove Theorems \ref{theorem:3pt-blowup} and \ref{theorem:Hirzebruch} in the following sections, we need to describe the setting and give some information whose nature we now explain. We also give some details about previous works on which it relies.

\textbf{A. Geometric setting.} We will first introduce the symplectic toric 4--manifold $(M,\omega)$ in which we are interested and describe the associated circle actions, moment map, and polytope. Then we will give topological information which will be useful: 
\begin{itemize}
\item the fundamental group of $\mathrm{Ham}(M,\omega)$, on  which the Seidel morphism is defined, and
\item the second homology group of $M$, which consists of generators of the quantum homology of $(M,\omega)$ (as a module over the Novikov ring).
\end{itemize}

\textit{Background for A. (see da Silva \cite{Cannas} for more details).} 
First, consider a Hamiltonian circle action on $(M, \omega)$. It is generated by a function $\phi: M \to \R$, called the \textit{moment map}, which is assumed to be normalized, that is, satisfying $\int_M \phi \, \omega^n=0$.

Now  $(M,\omega)$ is called \textit{toric} if it admits an effective action by a Hamiltonian torus $\bb T^2 \subset \mathrm{Ham}(M,\omega)$. We will denote by $\Phi$ the corresponding moment map and by $P = \Phi(M)$ the moment polytope. If $\eta$ is an outward primitive normal to the facet $D_\eta$ of $P$, we consider the associated Hamiltonian circle action, $\Gamma_\eta$,  whose moment map is $\phi := \langle \eta, \Phi ( \cdot ) \rangle$.\footnote{To lighten the notation, we will actually denote by $D_i$ and $\Gamma_i$, respectively,  the facet and the circle action associated to the normal $\eta_i$ (instead of $D_{\eta_i}$ and $\Gamma_{\eta_i}$).} 

Note that $\phi^{-1}(D_\eta)$ is a {\it semifree} maximum component for $\Gamma_\eta$, as the action is semifree (i.e. the  stabilizer of every point is trivial or the whole circle) on some neighborhood of $\phi^{-1}(D_\eta)$.

\textbf{B. The Seidel morphism.} In a second step, we will give the expression of the image of the aforementioned circle actions (the $\Gamma_\eta$'s) via the Seidel morphism, $\m S$.

\textit{Background for B. (see McDuff--Tolman \cite{McDuffTolman06} and Anjos--Leclercq \cite{AnjosLeclercq14}).}
We consider a toric 4--manifold $(M, \omega, \Phi)$ as above. To compute the image of a Hamiltonian circle action via the Seidel morphism, we pick a $\omega$--compatible, $\bb S^1$--invariant almost complex structure, $J$. The main case we are concerned with here is the Fano case. Recall that $(M,J)$ is said to be \textit{Fano} if any $J$--pseudo holomorphic sphere in $M$ has positive first Chern number. 

When this is the case, Theorem 1.10 from \cite{McDuffTolman06} or 4.5 from \cite{AnjosLeclercq14} tells us that the associated Seidel element consists of only one term (the one of highest order). More precisely,

\begin{theo}(\cite[Theorem1.10]{McDuffTolman06})\label{theo:main}
Let $(M, \omega, J, \Phi)$ be a compact Fano toric symplectic 4--manifold. Let $\eta$ be an outward primitive normal to the facet $D_\eta$ of the moment polytope $P$ and let $\Gamma_\eta$ be the associated Hamiltonian circle action. Then
$$ \m S (\Gamma_\eta)= [ F_{\rm{max}}] \otimes q t^{\phi_{\rm{max}}}$$
where $\phi$ is the moment map associated to $\Gamma_\eta$, $F_{\rm{max}}= \phi^{-1}(D_\eta)$ is the maximal fixed point component of $\phi$ and $\phi_{\rm{max}}= \phi (F_{\rm{max}})$.
\end{theo}

\textbf{C. The quantum homology of $(M,\omega)$.} The computation of the Seidel elements $\m S(\Gamma_\eta)$ in Step \textbf{B.} also gives us explicit relations involving the quantum product. This allows us to complete the description of the quantum homology as an algebra. 
Since the generators of $\pi_1(\mathrm{Ham}(M,\omega))$ can be expressed in terms of the $\Gamma_\eta$'s, this also gives us the image of the Seidel morphism so that, by understanding $\mathrm{im}(\m S) \subset QH_* (M, \omega)^\times$, we can prove Theorems \ref{theorem:3pt-blowup} and \ref{theorem:Hirzebruch}.

\textit{Background for C. (see McDuff--Tolman \cite[Section 5.1]{McDuffTolman06} for the general setting).} 
Let us recall how to obtain the quantum homology algebra in our specific setting. Let $D_1, \hdots, D_n$ be the facets of $P$ and $\eta_1, \hdots, \eta_n \in \R^2$ the respective outward primitive integral normal vectors. Let $C$ be the set of \textit{primitive sets}, i.e subsets $I= \{ i_1,i_2 \} \subset \{ 1, \hdots, n \}$ such that $D_{i_1} \cap D_{i_2} = \emptyset$. Let $u_i= [D_i] \otimes q$. There are two linear relations: 
\begin{equation*}
\sum_{i=1}^{n} \langle (1,0), \eta_i \rangle \, u_i = 0  \quad \mbox{and} \quad \sum_{i=1}^{n} \langle (0,1), \eta_i \rangle \, u_i = 0
\end{equation*}
which generate the ideal of linear relations $\Lin (P)$ in $\Q[u_1, \hdots, u_n]$. Moreover, relations between the normal vectors $\eta_i$'s yield equations satisfied by the corresponding Seidel elements $\m S(\Gamma_{i})$. Using these, it is then possible to exhibit  the quantum product $u_{i_1} * u_{i_2}$, for every primitive set $\{ i_1, i_2 \}$, as a linear combination of the classes $p$ (the class of a point), $\mathbbm{1}$ (the fundamental class), and $u_i$: $f_{i_1i_2} = (\alpha p  \otimes q^2 + \beta \mathbbm{1} + \sum \alpha_i u_i) \, t ^\gamma$ for some $\alpha, \beta, \alpha_i \in \Z$ and $\gamma \in \R$. 
Then, the Stanley--Reisner ideal is defined by 
\begin{equation*}
\SR_Y(P) = \langle  u_{i_1} * u_{i_2}- f_{i_1i_2} \ | \ \{ i_1, i_2 \}  \in C \rangle.
\end{equation*}
Finally, there is an isomorphism of $\Piuniv$--algebras
\begin{equation}\label{quantum-homology}
 \QH_* (M, \omega) \simeq \Q[u_1, \hdots, u_n] \otimes \Piuniv / (\Lin(P) + \SR_Y (P)) \,.
 \end{equation}


\section{Hirzebruch surfaces}
\label{sec:hirzebruch-surfaces}

We proceed in two steps as the ``even'' and ``odd'' Hirzebruch surfaces have to be dealt with separately. In the whole section, we follow the notation and conventions 
used in \cite{AnjosLeclercq14} (in particular in Section 5.3), most of them being recalled in Section \ref{sec:background} above.

\subsection{Even Hirzebruch surfaces}
\label{sec:even-hirz-surf}

Recall that the toric ``even'' Hirzebruch surfaces $(\F_{2k},\omega_{\mu})$, $0\leq k\leq \ell$ with $\ell \in \N$ and $\ell < \mu \le \ell+1$, can be identified with the symplectic manifolds $M_\mu=(S^{2}\times S^{2},\omega_\mu)$ where $\omega_\mu$ is  the split symplectic form with area $\mu \geq 1$ for the first $S^2$--factor, and with area 1 for the second factor. The moment polytope of $\F_{2k}$ is
$$P_{2k}=\left\{ (x_1,x_2) \in \R^2 \mid 0 \leq x_1 \leq 1, \ x_2+kx_1 \geq 0, \; \ x_2-kx_1 \leq \mu-k\right\}.$$ 
Let $\Lambda^{2k}_{e_1}$ and $\Lambda^{2k}_{e_2}$ represent the circle actions whose moment maps are, respectively, the first and second components of the moment map associated to the torus action $T_{2k}$ acting on $\F_{2k}$. We will also denote by $\Lambda^{2k}_{e_1}$ and $\Lambda^{2k}_{e_2}$ the corresponding generators in $\pi_1(T_{2k})$.

It is well known (see e.g \cite[Theorem 1.1 or Corollary 2.7]{AMcD}) that  for $k=0$,   $\pi_1(\ham(\bb F_{0},\omega_\mu))=\bb Z/2 \oplus \bb Z/2$ and  that for $k \geq 1$,  $\pi_1(\ham(\bb F_{2k},\omega_\mu))= \bb Z/2 \oplus \bb Z/2 \oplus \bb Z$. Moreover, the authors explain in \cite{AMcD} (see Section 2.5 and in particular Lemma 2.10) that the $\bb Z/2$ terms of the fundamental groups are respectively generated by $ \Lambda^{0}_{e_1}$ and $ \Lambda^{0}_{e_2}$, while the generator of the additional $\Z$ term is $ \Lambda^{2}_{e_1}$.

Let $B=[S^2 \times \{p\}]$ and $F=[\{p\} \times S^2] \in H_2( S^2 \times S^2; \Z)$ and denote $u=B\otimes q$  and $v=F \otimes q$ where $q$ is the degree 2 variable entering into play in the definition of $\Pi = \Piuniv[q,q^{-1}]$ and $\Pi^{\mathrm{univ}}$ the ring of generalised Laurent series defined by \eqref{eq:PiUniv}. 

We now gather from \cite{AnjosLeclercq14} the results we will need for the proof of Theorem \ref{theorem:Hirzebruch} in this case. First, in \cite[Section 5.3]{AnjosLeclercq14}, we computed the image of the generators $ \Lambda^{0}_{e_1}$, $ \Lambda^{0}_{e_2}$, and $ \Lambda^{2}_{e_1}$ by the Seidel morphism, $\m S$. Namely we obtained:
\begin{align*}\label{eq:Seidel}
  \m S(\Lambda^{0}_{e_1}) &= B \otimes qt^{\frac{1}{2}}=ut^{\frac{1}{2}}, \quad \m S(\Lambda^{0}_{e_2})=F \otimes qt^{\frac{\mu}{2}}=vt^{\frac{\mu}{2}}, \quad \mbox{and}
 \end{align*}
 \begin{equation}\label{Seidelelement}
  \m S(\Lambda_{e_1}^2) = (B+F) \otimes qt^{\frac{1}{2} -\epsilon}=(u+v)t^{\frac{1}{2} -\epsilon}  \quad \mbox{with } \epsilon=\frac{1}{6\mu} .
  \end{equation}
  Note that the circle action $\Lambda_{e_1}^2$  acts on the second Hirzebruch surface $\F_2$ and the almost complex structure in this case is not Fano, because the class $B-F$ is represented by a pseudo-holomorphic sphere and its first Chern number vanishes. Nevertheless, by Theorem 4.4 in \cite{AnjosPinsonnault13}, the Seidel element of this action still does not contain any lower order terms. 
 
The computation of the Seidel elements associated to each one of the facets of the polytope yield the following quantum product identities
 \begin{equation}\label{quantum-products}
  F * F = \mathbbm{1} \otimes q^{-2}t^{-\mu}, \quad B * B = \mathbbm{1} \otimes q^{-2}t^{-1}, \quad \mbox{and} \quad F*B=p \,
  \end{equation}
 so $S(\Lambda^{0}_{e_1})^2=S(\Lambda^{0}_{e_2})^2= \mathbbm{1}$.
Finally recall that, thanks to \cite[Proposition 5.1]{AnjosLeclercq14} (see \eqref{quantum-homology} in our setting), we were able to express  the (small) quantum homology  algebra as 
$$QH_*(\bb F_{2k},\omega_\mu) \simeq \Pi^{\mathrm{univ}}[u,v] / \langle u^2=t^{-1}, v^2=t^{-\mu} \rangle \,.$$  
From  \eqref{Seidelelement} and \eqref{quantum-products}, it is now easy to check that the inverse of $S(\Lambda_{e_1}^2)$ is given by
\begin{equation}\label{inverse}
\m S(\Lambda_{e_1}^2)^{-1} =(B-F) \otimes q \, \frac{t^{\frac{1}{2} +\epsilon}}{1-t^{1-\mu}}=(u-v) \, \frac{t^{\frac{1}{2} +\epsilon}}{1-t^{1-\mu}} \,.
\end{equation}
Let us now prove the theorem.

\begin{proof}[Proof of Theorem \ref{theorem:Hirzebruch} for even Hirzebruch surfaces]
  Since $\Lambda^{0}_{e_1}$ and $\Lambda^{0}_{e_2}$ are of order 2, any element in $\pi_1(\ham(\bb F_{2k},\omega_\mu))$ is of the form $\varepsilon_1\Lambda^{0}_{e_1} + \varepsilon_2\Lambda^{0}_{e_2} + \ell \Lambda^{2}_{e_1}$, with $\varepsilon_1$ and $\varepsilon_2$ in $\{ 0,1 \}$ and $\ell \in \bb Z$. Moreover, it is in the kernel of $\m S$ if and only if $\m S(\Lambda^{2}_{e_1})^{-\ell} = \m S(\Lambda^{0}_{e_1})^{\varepsilon_1}\m S(\Lambda^{0}_{e_2})^{\varepsilon_2}$ which is equivalent to the fact that $\m S(\Lambda^{2}_{e_1})^{-\ell}$ is either $u$, $v$, or $uv$, up to a power of $t$.

Let $\ell' \in \bb N\setminus \{0\}$, and expand the $\ell'$--th power of $\m S(\Lambda^{2}_{e_1})$ (whose expression is recalled in \eqref{Seidelelement} above) thanks to the binomial theorem to get
$$ S(\Lambda^{2}_{e_1})^{\ell'} = \sum_{k=0}^{\ell'} \left( \!\! \begin{array}{c} \ell' \\ k \end{array}  \!\! \right) u^k v^{\ell'-k} t^{(\frac12 -\epsilon)\ell'} \,.$$
The identities $u^2=t^{-1}$ and $v^2=t^{-\mu}$ ensure that $ S(\Lambda^{2}_{e_1})^{\ell'}$ is of the form $C_1\cdot u + C_2 \cdot v$ if $\ell'$ is odd, or $C_1 + C_2 \cdot uv$ otherwise, where (in both cases) $C_1$ and $C_2$ are linear combinations of powers of $t$ with positive rational coefficients (hence non zero). Thus $\varepsilon_1\Lambda^{0}_{e_1} + \varepsilon_2\Lambda^{0}_{e_2} + \ell \Lambda^{2}_{e_1} \notin \ker(\m S)$ for any $\varepsilon_1$ and $\varepsilon_2$ in $\{ 0 , 1 \}$ and $\ell < 0$. 

We proceed along the same lines for a positive $\ell$ : $\m S(\Lambda^{2}_{e_1})^{-\ell}$ is, by the binomial theorem together with \eqref{inverse}, of the form $$\frac{C'_1\cdot u - C'_2 \cdot v}{(1-t^{1-\mu})^{\ell}}\quad \mbox{or} \quad  \frac{C'_1 - C'_2 \cdot uv}{(1-t^{1-\mu})^{\ell}}$$ which shows that $\varepsilon_1\Lambda^{0}_{e_1} + \varepsilon_2\Lambda^{0}_{e_2} + \ell \Lambda^{2}_{e_1} \notin \ker(\m S)$ for any $\ell > 0$ either. 

This implies that the only elements of $\pi_1(\ham(\bb F_{2k},\omega_\mu))$ which could be in $\ker(\m S)$ are of the form $\varepsilon_1\Lambda^{0}_{e_1} + \varepsilon_2\Lambda^{0}_{e_2}$ so that in the end $\ker(\m S)=\{0\}$.
\end{proof}

\subsection{Odd Hirzebruch surfaces}\label{sec:odd-hirz-surf} 
Similarly, ``odd'' Hirzebruch surfaces $(\F_{2k-1},\omega_{\mu}')$, $1\leq k\leq \ell$ with $\ell \in \N$ and  $\ell < \mu \le \ell+1$, can be identified with the symplectic manifolds $(\PbP,\omega_\mu')$ where the symplectic area of the exceptional divisor is  $\mu >0$ and the area of the projective line is $\mu +1$. Its moment polytope is 
\begin{align*}
  \left\{ (x_1,x_2) \in \R^2 \left| 
      \begin{array}{l}
        0 \leq x_1+x_2 \leq 1, \ x_2(k-1)+kx_1 \geq 0,\\ \ kx_2+(k-1)x_1 \geq k -\mu-1
      \end{array}
\right. \right\} \,.
\end{align*}
Let $\Lambda^{2k-1}_{e_1}$ and $\Lambda^{2k-1}_{e_2}$ represent the circle actions whose moment maps are, respectively, the first and the second component of the moment map associated to the torus action $T_{2k-1}$ acting on $\F_{2k-1}$. As before, we will also denote by $\Lambda^{2k-1}_{e_1}$ and $\Lambda^{2k-1}_{e_2}$ the generators of $\pi_1(T_{2k-1})$. 

Similarly to the even case the fundamental group of $(\F_{2k-1},\omega_{\mu}')$ is computed in \cite[Theorem 1.4 or Corollary 2.7]{AMcD}. More precisely, $\pi_1(\ham(\F_{2k-1},\omega_{\mu}'))= \Z \langle \Lambda^{1}_{e_1} \rangle$  for all $k \geq 1$, that is,  $\Lambda^{1}_{e_1}$ is the generator of the fundamental group as explained in \cite[Section 2.5]{AMcD} (see in particular Lemma 2.11). So, in order to prove that the Seidel morphism is injective, we only need to show that the order of $\m S(\Lambda^{1}_{e_1})$ in $QH_*(\bb F_{2k+1},\omega_\mu')$ is infinite.

We now need to expand Remark 5.6 of \cite{AnjosLeclercq14} (which quickly dealt with the odd case), along the lines of \cite[Section 5.3]{AnjosLeclercq14} (where we focused with more details on the even case).
Let $B \in H_2(\PbP; \Z)$ denote the homology class of the exceptional divisor with self intersection $-1$ and  $F$  the class of the fiber of the fibration $\PbP \to S^2$. 
If we set $u_1=(B+F) \otimes q$, $u_2=u_4= F \otimes q$, and $u_3= B \otimes q$, clearly the additive relations are given by 
\begin{equation}\label{additive}
u_2 = u_4 \quad  \mbox{and} \quad  u_1=u_2 +u_3.
\end{equation}
The normal vectors to the moment polytope of $\F_1$ are given by $\eta_1=(1,1)$,  $\eta_2=(0,-1)$,  $\eta_3=(-1,-1)$, and $  \eta_4=(-1,0)$. We denote by $\Gamma_i$ the actions associated to $\eta_i$. 

As explained in Section \ref{sec:background}, since $\F_1$ is Fano, it follows from \cite[Theorem 1.10]{McDuffTolman06} that the Seidel elements associated to the $\Gamma_i$'s are given by 
  \begin{gather*}
  \m S(\Gamma_1) =(B +F) \otimes qt^{1+\mu -2\varepsilon}=u_1t^{1+\mu -2\varepsilon} , \quad \m S(\Gamma_2) = \m S(\Gamma_4)=F \otimes qt^{\varepsilon}=u_2t^{\varepsilon}, \\
  \m S(\Gamma_3) =B \otimes qt^{2\varepsilon-\mu}=u_3t^{2\varepsilon-\mu} \quad \mbox{with }  \varepsilon= \frac{3 \mu^2 + 3\mu +1}{3(1+2\mu)}.
 \end{gather*}

The relation $\eta_1+\eta_3=0$ yields $\m S(\Gamma_1) *  \m S(\Gamma_3) = \mathbbm{1}$, that is, $ B* (B+F) \otimes q^2 t = \mathbbm{1}$. Similarly,  since $\eta_2 +\eta_4=\eta_3$ it follows that $\m S(\Gamma_2) *  \m S(\Gamma_4) =  \m S(\Gamma_3) $ which is equivalent to $F*F = B \otimes q^{-1} t^{-\mu}$. Therefore the primitive relations are given by 
 \begin{equation}\label{primitive}
  u_1 u_3=t^{-1}\quad  \mbox{and} \quad u_2u_4=u_3t^{-\mu}.
  \end{equation}
Now, following Step \textbf{C.} of Section \ref{sec:background} above, we set $u= F\otimes q$ and deduce from the relations  \eqref{additive} and  \eqref{primitive} that
\begin{align}\label{QH-oddHirzebruch}
QH_*(\bb F_{2k+1},\omega_\mu') = \Pi^{\mathrm{univ}}[u] / (u^4 t^{2\mu} + u^3t^{\mu} - t^{-1}) \,.
\end{align}

Note that  the generator of $\pi_1(\ham(\F_{2k-1},\omega_{\mu}'))$, $\Lambda^{1}_{e_1}$, is the action associated to the vector $(1,0)$. We thus get that $\m S(\Lambda^{1}_{e_1}) = \m S(\Gamma_4)^{-1}$.

Now we can proceed with the proof of the theorem.

\begin{proof}[Proof of Theorem \ref{theorem:Hirzebruch} for odd Hirzebruch surfaces]
From the discussion above, we see that $\m S(\Lambda^{1}_{e_1})^{-1}=\m S(\Gamma_4) = u t^{\varepsilon}$. So, in order to show that Seidel's morphism is injective we only need to show that, for any $\ell \in \bb N \setminus \{0\}$, $\m S(\ell\Lambda^{1}_{e_1})^{-1} = u^\ell t^{\ell\varepsilon} \neq 1$.

First notice that the polynomial $M(u)=u^4 t^{2\mu} + u^3t^{\mu} - t^{-1} \in \Pi^{\mathrm{univ}}[u]$ in \eqref{QH-oddHirzebruch} above has invertible main coefficient, so that for any positive integer $\ell$, there exist uniquely determined polynomials $Q_\ell$ and $R_\ell$ such that $u^\ell t^{\ell\varepsilon}-1 = M(u)Q_\ell(u) + R_\ell(u)$ and degree of $R_\ell$ is less than degree of $M$. 

Assume that Seidel's morphism is not injective: then there exists $\ell_0 \in \bb N \setminus \{0\}$ such that $R_{\ell_0} = 0$. To determine the polynomial $Q_{\ell_0}$, we proceed to the long division of $u^{\ell_0} t^{{\ell_0}\varepsilon}-1$ by $M$ which consists in a finite number of (at most ${\ell_0} - 3$) steps. This ensures that the coefficients of $Q_{\ell_0}$ are \textit{finite} linear combinations of powers of $t$ (with rational coefficients). Thus $Q_{\ell_0}$ induces a polynomial in $\bb Q[u]$ when $t$ is set to 1, $Q^1_{\ell_0}$, which satisfies $u^{\ell_0} - 1 = (u^4 + u^3 -1) Q^1_{\ell_0}(u)$ in $\bb Q[u]$. Since the roots of $u^4  + u^3 - 1$ are not roots of unity, we get a contradiction.
Thus, there is no positive integer ${\ell_0}$ so that $u^{\ell_0} t^{{\ell_0}\varepsilon} = 1$ which concludes the proof.
\end{proof}


\section{2--point blow-ups of \texorpdfstring{$S^2\times S^2$}{Lg}}
\label{sec:2-point-blowup}

We now consider the manifold obtained from $(M_\mu, \omega_\mu) = (S^2 \times S^2,\omega_\mu)$ (see Section \ref{sec:even-hirz-surf}) by performing two successive symplectic blow-ups of capacities $c_1$ and $c_2$  with  $0 < c_2 \leq c_1 < c_1 + c_2 \leq 1 \leq \mu$, which we denote by $(\Mucc, \omucc)$. Let $B$, $F \in H_2(\Mucc;\Z)$ be the homology classes defined by  $B=[S^2 \times \{p\}]$, $F=[\{p\} \times S^2]$ and  let $E_i\in H_2(\Mucc;\Z)$ be the exceptional class corresponding to the blow-up of capacity $c_i$. 

\begin{rema}
There is an alternative description of this manifold as the 3--point blow-up of $\bbcp^2$. Indeed, consider $\X_3={\bbcp^2\#\,3\overline{ \bbcp}\,\!^2}$ equipped with the symplectic form $\omega_{\nu;\delta_{1},\delta_2,\delta_{3}}$  obtained from the symplectic blow-up of $(\bbcp^{2},\omega_{\nu})$ at 3 disjoint balls of capacities $\delta_1, \delta_2$ and $\delta_3$, where $\omega_{\nu}$ is the standard Fubini--Study form on $\bbcp^{2}$ rescaled so that $\omega_{\nu}(\bbcp^{1})=\nu$. Let $\{L,V_{1},V_2,V_{3}\}$ be the standard basis of $H_{2}(\X_{3};\Z)$ consisting of the class $L$ of a line together with  the classes $V_{i}$ of the exceptional divisors. It is well known that $\X_3$ is diffeomorphic to $\Mucc$. The diffeomorphism $\X_{3} \to \Mucc$ can be chosen to map the ordered basis $\{L,~ V_{1},~ V_{2},~ V_{3}\}$ to $\{B+F-E_{1},~ B-E_{1},~ F-E_{1},~E_{2}\}$. 
When one considers this birational equivalence in the symplectic category, uniqueness of symplectic blow-ups implies that $(\X_{3}, \omega_{\nu; \delta_{1}, \delta_{2}, \delta_{3}})$ is symplectomorphic, after rescaling, to $M_{\mu}$ blown--up with capacities $c_{1}$ and $c_{2}$, where $\mu=(\nu-\delta_{2})/(\nu-\delta_{1})$, $c_{1}=(\nu-\delta_{1}-\delta_{2})/(\nu-\delta_{1})$, and $c_{2}=\delta_{3}/(\nu-\delta_{1})$.
In \cite[Section 2.1]{AnjosPinsonnault13}, it is explained why it is sufficient to consider values of $c_1$ and $c_2$ in the range above:  $0 < c_2 \leq c_1 < c_1 + c_2 \leq 1 \leq \mu$.
\end{rema}

The quantum algebra of $(\Mucc, \omucc)$ was computed by Entov--Polterovich in \cite{EntovPolterovich08} (as $(\X_3,\omega_{\nu;\delta_{1},\delta_2,\delta_{3}})$, see the proof of Proposition 4.3). More precisely, setting $u=(F-E_2) \otimes q$ and $v=(B-E_2)\otimes q$, they proved that: 
\begin{lemm} As a $ \Pi^{\mathrm{univ}}$--algebra we have 
$$QH_*(\Mucc,\omucc) \cong \Pi^{\mathrm{univ}}[u,v] / I_{\mu, c_1, c_2}$$ where $I_{\mu, c_1, c_2}$ is the ideal generated by
\begin{align*}
&  u^2v^2 + u^2vt^{-c_2} = vt^{-\mu-c_2} + t^{c_1-\mu-1-c_2} \mbox{ and }\; \\
&  u^2v^2 + uv^2t^{-c_2} = ut^{-1-c_2} + t^{c_1-\mu-1-c_2} \,.
\end{align*}
\end{lemm}

We recall here parts of this computation, \textit{using the formalism of \cite{AnjosLeclercq14}}, as they will be needed below to understand the proof of the non-injectivity result stated as Theorem \ref{theorem:3pt-blowup}. These parts correspond to Steps \textbf{B.} and \textbf{C.} of Section \ref{sec:background}. 

\begin{proof}[Sketch of proof]
 Consider $(\Mucc,\omucc)$ endowed with the standard action of the torus $T=S^1 \times  S^1$ for which the moment polytope is given by 
\begin{equation*}
P =  \left\{ (x_1,x_2) \in \R^2 \mid 0 \leq x_2 \leq \mu, \, -1 \leq x_1 \leq 0,
  \, c_1 \leq x_2 -x_1 \leq \mu+1 -c_2  \right \}
\end{equation*}  so the primitive outward normals to $P$ are as follows:
\begin{equation*} \eta_1=(0,1), \, \eta_2=(1,0), \,
  \eta_3=(1,-1), \, \eta_4=(0,-1), \, \eta_5=(-1,0), \ \mbox{and} \ \eta_6=(-1,1).
\end{equation*} 
The Delzant construction  gives a method  to obtain, from the polytope $P$,  the symplectic manifold $(\Mucc, \omucc)$ with the toric action $T$: First consider the standard action of the torus $\mathbb T^6$  on $\bb C^6$  and then perform a symplectic reduction at a regular level of that action (for more details see for example \cite[Section 29]{Cannas}). Then the normalised moment map $\Phi: \Mucc \to \R^2$ of the remaining $T$ action, obtained through the Delzant construction,  is given by
$$ \Phi(z_1, \hdots, z_6)=\left(-\frac12 |z_2|^2 + \epsilon_1, -\frac12 |z_1|^2 + \mu - \epsilon_2 \right),  \quad z_i \in \C, $$
where $\epsilon_1$ and $\epsilon_2$ are given by the symplectic parameters $\mu$, $c_1$, and $c_2$ as:
\begin{equation}\label{epsilons}
\epsilon_1=\frac{c_1^3+3c_2^2-c_2^3 -3\mu}{3(c_1^2+c_2^2-2\mu)} \quad \mbox{and} \quad  \epsilon_2=\frac{c_1^3 - c_2^3 +3c_2^2\mu-3\mu^2}{3(c_1^2+c_2^2-2\mu)} \,.
\end{equation}
     Moreover, the homology classes $A_i= [\Phi^{-1}(D_i)]$ of the pre-images of the corresponding facets $D_i$ are:  
$A_1=F-E_2 $,  $A_2=B-E_1$, $A_3=E_1$,  $A_4=F-E_1$, $A_5=B-E_2$, and 
$A_6=E_2$.

For $1 \leq i\leq 6$, let $\Gamma_i$ be the circle action associated to the primitive outward normal $\eta_i$.
Since the toric  complex structure on $\Mucc$ is Fano and $T$--invariant, it follows from \cite[Theorem 1.10]{McDuffTolman06} or  \cite[Theorem 4.5]{AnjosLeclercq14}  (recalled as Theorem \ref{theo:main} in Section \ref{sec:background}) that the Seidel elements associated to the $\Gamma_i$'s are given by the following expressions
\begin{align*}
\m S(\Gamma_1) & =(F-E_2) \otimes qt^{\mu-\epsilon_2}\,,& \quad \quad  & \m S(\Gamma_2)  = (B-E_1) \otimes q t^{\epsilon_1}\,,\\
\m S(\Gamma_3) & =E_1 \otimes qt^{\epsilon_1+\epsilon_2-c_1}  \,,&\quad \quad  & \m S(\Gamma_4)   =(F-E_1)\otimes q t^{\epsilon_2}\,,\\ 
\m S(\Gamma_5)  & =(B-E_2)\otimes q {t^{1-\epsilon_1}}\,, &\quad \quad  & \m S(\Gamma_6)   =E_2 \otimes qt^{\mu+1 -c_2- \epsilon_1-\epsilon_2} \,.
\end{align*}
There are nine primitive sets: $\{1,3\}$, $\{1,4\}$, $\{1,5\}$, $\{2,4\}$, $\{2,5\}$, $\{2,6\}$, $\{3,5\}$, $\{3,6\}$, and $\{4,6\}$ which yield nine multiplicative relations (which form the Stanley--Riesner ideal) that, combined with the two linear relations $(A_5=A_1+A_2-A_4$ and $A_6=A_3+A_4-A_1)$, give the desired result as explained in Step \textbf{C.} of Section \ref{sec:background} above.
\end{proof}
Assume from now on that $\mu=1$. Recall from \cite[Theorem 1.1]{AnjosPinsonnault13} that if $c_2 < c_1$ then $$\pi_1(\ham(M_{1,c_1,c_2}, \omega_{1,c_1,c_2})) \simeq \bb Z \langle x_0,x_1,y_0,y_1,z \rangle \simeq \bb Z^5$$ where the generators $x_0,x_1,y_0,y_1,z$ correspond to circle actions contained in maximal tori  of the Hamiltonian group. In particular, the generators in which we will be most interested are $x_0=\Gamma_2$ and $y_0=\Gamma_1$ where the $\Gamma_i$'s are defined in the proof  of the  lemma above. 

\begin{rema}
In order to understand the remaining generators, consider the two toric manifolds given by the polytopes in Figure \ref{polytope}. We denote by $\{ x_{0,i}, y_{0,i}\}$ the generators in $\pi_1{(T_i)}$,  where $T_i$, $i=1,2$, represent the two torus actions in  this figure and the generators $\{ x_{0,i}, y_{0,i}\}$ correspond to the circle actions whose moment maps are, respectively, the first and second components of the moment map associated to each one of the toric actions. 
It was shown in \cite[Lemma 4.5]{AnjosPinsonnault13} that $x_1=x_{0,1}$, $z=y_{0,2}$, and $y_1=y_{0,1}-x_1=z-x_{0,2}$.  

Note that the case $c_1=c_2$ is an interesting limit case in terms of the topology of the Hamiltonian group since $y_1$ disappears. For more details see \cite[Section 5.1]{AnjosPinsonnault13}.   
\end{rema}

\begin{figure}
\begin{tikzpicture} 
    \draw[->]  (0.2,0)--(-2.2,0) node[left] {\footnotesize $x$};
    \draw[->] (0,-0.2) -- (0,3.2) node[above] {\footnotesize $y$};
    \draw[domain=-1.9:0] plot (\x,{2.8 +0* \x});
    \draw[domain=-1:-0.4] plot (\x,{1+1 * \x}); 
    \draw[domain=-0.4:0] plot (\x,{1.4+2 * \x}); 
    \draw (-1.9,0) -- (-1.9,2.8);
    
    \node at (0.2,2.8) {\footnotesize$\mu$};
    \node at (0.6,1.4) {\footnotesize $c_1+c_2$};
    \node at (0.6,0.6) {\footnotesize $c_1 -c_2$};
    \node at (-0.4,-0.3) {\footnotesize $-c_2$};
    \node at (-1.2,-0.3) {\footnotesize $-c_1$};
    \node at (-2,-0.3) {\footnotesize $-1$};
   
    \draw[dashed] (-0.4,0.6) -- (0,0.6);
    \draw[dashed] (-0.4,0.6) -- (-0.4,0);
 
    \draw[->]  (5.2,0)--(2.8,0) node[left] {\footnotesize $x$};
    \draw[->] (5,-0.2) -- (5,3.2) node[above] {\footnotesize $y$};
    \draw[domain=3.1:5] plot (\x,{2.8 +0* \x});
    \draw[domain=4.4:5] plot (\x,{-4+1 * \x}); 
    \draw[domain=3.6:4.4] plot (\x,{-1.8+0.5 * \x}); 
    \draw (3.1,0) -- (3.1,2.8);
    
    \node at (5.2,2.8) {\footnotesize$\mu$};
    \node at (5.3,1) {\footnotesize $c_1$};
    \node at (5.3,0.4) {\footnotesize $c_2$};
    \node at (4.4,-0.3) {\footnotesize $-c_1+c_2$};
    \node at (3.8,-0.7) {\footnotesize $-c_1-c_2$};
    \node at (3,-0.3) {\footnotesize $-1$};
   
    \draw[dashed] (4.4,0.4) -- (5,0.4);
    \draw[dashed] (4.4,0.4) -- (4.4,0);
    \draw[dashed] (3.6,0) -- (3.6,-0.5);
  \end{tikzpicture}
  
\caption{$(\Mucc, \omucc)$ with toric actions $T_1$ and $T_2$.}
\label{polytope}
\end{figure}

To prove Theorem \ref{theorem:3pt-blowup}, we will now show that
\begin{prop}
  The class of $2(x_0+y_0)$ belongs to $\ker(\m S)$ if and only if $c_1 = c_2$.
\end{prop}

\begin{proof}
  From the computation of the Seidel elements in the proof of the lemma above one gets that in the general case (by which we mean for all $\mu \geq 1$): $\m S(\Gamma_1)  =ut^{\mu-\epsilon_2}$ and $\m S(\Gamma_5)  =v {t^{1-\epsilon_1}}$. As the Seidel elements are \textit{invertible} quantum classes, this yields invertibility of $u$ and $v$. Note that $\m S(x_0)=\m S(\Gamma_2)=\m S(\Gamma_5)^{-1}=v^{-1}t^{\varepsilon_1 -1}$ and $\m S(y_0)=\m S(\Gamma_1)=ut^{\mu-\varepsilon_2}$.
  
Since $\mu \geq 1 > c_2^2$, it is straightforward to deduce from \eqref{epsilons} that $\varepsilon_1 = \varepsilon_2$ if and only if $\mu = 1$: we now restrict our attention to this case and denote by $\varepsilon$ the common value of  $\varepsilon_1 = \varepsilon_2$. By invertibility of $u$ and $v$, the fact that $2(x_0+y_0)$ belongs to $\ker(\m S)$ is equivalent to $u^2 = v^2$, since 
$$\m S (2(x_0+y_0))= \m S(x_0)^2 *\m S(y_0)^2= v^{-2}t^{\epsilon-1}u^2 t^{1-\epsilon}=v^{-2}u^2 \,.$$
On the other hand, note that multiplying the first and second relations in $I_{1, c_1, c_2}$ by $v^{-1}t^{c_2}$ and $u^{-1}t^{c_2}$, respectively, these become equivalent to   
$$ u^2= t^{-1}+v^{-1}t^{c_1-2}-u^2vt^{c_2}  \quad \mbox{and} \quad v^2= t^{-1}+u^{-1}t^{c_1-2} -uv^2t^{c_2}$$
 so that $u^2=v^2$ is equivalent to $v^{-1}t^{c_1-2} -u^2vt^{c_2} = u^{-1}t^{c_1-2} -uv^2t^{c_2}.$

Multiplying both relations in $I_{1, c_1, c_2}$ by $t^{2c_2}$, we see that
\begin{align}\begin{split}\label{intermediate-equation}
  -u^2vt^{c_2} &= (u^2v^2t^{2c_2} - t^{c_1+c_2-2}) - vt^{c_2-1}, \quad\mbox{and} \\
  -uv^2t^{c_2} &= (u^2v^2t^{2c_2} - t^{c_1+c_2-2}) - ut^{c_2-1}
\end{split}\end{align}
so that we can replace $u^2vt^{c_2}$ and $uv^2t^{c_2}$ in the previous equality to obtain
\begin{align}\label{equation}
  u^2 = v^2  \;\Longleftrightarrow\; v^{-1}t^{c_1-1} + ut^{c_2} = u^{-1}t^{c_1-1} + vt^{c_2} \,.
\end{align}
Finally, \eqref{intermediate-equation} also induces, by subtracting one from the other, the equality 
$(u^2v -uv^2)t^{-c_2}=(v-u)t^{-1-c_2}$  which is equivalent to $(v^{-1} - u^{-1})t^{-1} = v - u$. 
Using these together with \eqref{equation} we conclude that $u^2 = v^2$ if and only if $(u-v)(t^{c_1}-t^{c_2})=0$ which is equivalent to $c_1 = c_2$ since otherwise $t^{c_1}-t^{c_2}$ would be invertible.
\end{proof}

\end{document}